\documentclass{amsart}

\usepackage{amssymb, graphicx, amsrefs}
\usepackage{color}
\usepackage{mathtools}
\usepackage{pgfplots}

\copyrightinfo{2024}{Sean Ku}

\newtheorem{theorem}{Theorem}[section]
\newtheorem{lemma}[theorem]{Lemma}
\newtheorem{proposition}[theorem]{Proposition}
\newtheorem{corollary}[theorem]{Corollary}

\theoremstyle{definition}
\newtheorem{definition}[theorem]{Definition}

\theoremstyle{remark}
\newtheorem{remark}[theorem]{Remark}
\newcommand{\RN}[1]{%
  \textup{\uppercase\expandafter{\romannumeral#1}}%
}
\numberwithin{equation}{section}
\usepackage[margin=1.5in]{geometry}

\usepackage{amsmath}%
\usepackage{amsfonts}%
\usepackage{braket}

\newcommand\norm[1]{\left\lVert#1\right\rVert}

\usepackage{breqn}
\usepackage{float}
\newcommand\numberthis{\addtocounter{equation}{1}\tag{\theequation}}

\providecommand{\eat}[1]{}

\newcommand{\Hm}[1]{\leavevmode{\marginpar{\tiny%
$\hbox to 0mm{\hspace*{-0.5mm}$\leftarrow$\hss}%
\vcenter{\vrule depth 0.1mm height 0.1mm width \the\marginparwidth}%
\hbox to 0mm{\hss$\rightarrow$\hspace*{-0.5mm}}$\\\relax\raggedright
#1}}}

\begin{document}

\title[Essential Self-Adjointness Characterizations]{Essential Self-Adjointness of Semi-Bounded Schr\"{o}dinger Operators on Birth-Death Chains }

\author{Sean Ku}
\address{S.~Ku, New York University, Courant Institute of Mathematical 
Sciences, 
251 Mercer St, New York, NY 10012}
\email{sk8980@nyu.edu}


\date{\today}
\thanks{S.~K.~is supported by the National Science Foundation funded Queens Experiences in 
Discrete Mathematics (QED) REU program, Award Number DMS 2150251.}

\begin{abstract}
 We study the essential self-adjointness of semi-bounded Schr\"{o}dinger operators on birth-death chains. First, we offer a general characterization which originates from studying a second order linear recurrence with variational coefficients which comes from the Schr\"{o}dinger operator. As this characterization is algebraically complicated, we  present an additional theorem discussing the failure of essential self-adjointness. Finally, we study two specific cases of solutions to equations involving the Schr\"{o}dinger operator over birth-death chains and derive explicit formulas in these cases. 
\end{abstract} 

\maketitle
\tableofcontents

\section{Introduction}  
     Essential self-adjointness is an extensively studied topic in mathematical physics. This property naturally arises in quantum mechanics and quantum observables \cite{Cintio_2021}. In particular, of the variety of operators studied, elliptic differential operators have been studied in the form of the Laplacian and Schr\"{o}dinger operators over settings such as the real numbers and smooth manifolds \cite{Grigoryan_2009, Reed_Simon_1975}.  \smallskip 
    
 \indent The aim of this note is to study the essential self-adjointness of semi-bounded Schr\"{o}dinger operators on birth-death chains, i.e., graphs with the structure of the natural numbers. As this analysis requires the understanding of solutions to equations involving the Schr\"{o}dinger operator, we introduce techniques and formulas to understand such solutions over birth-death chains. We introduce two methods to inspect essential self-adjointness of  Schr\"{o}dinger operators. The first method requires purely algebraic calculations and offers a full classification in terms of the graph structure. It turns out that such a method provides a classification which, while true, yields a convoluted and impractical expression to evaluate essential self-adjointness. The second method transforms the Schr\"{o}dinger operator problem to a matrix recurrence relation, which allows for the use of standard tools of linear algebra. We also provide alternative methods in studying solutions to the equation involving Schr\"{o}dinger operator through generating functions. \smallskip 

The rest of this note is structured in the following manner. Section \ref{Section 2} introduces the necessary definitions and background. Section \ref{Section: Characterization of ESA from Solving Linear Recurrence} derives a full characterization of essential self-adjointness of semi-bounded Schr\"{o}dinger operators on birth-death chains. In Section \ref{Section: Additional Results}, we derive an additional result on failure of essential self-adjointness. Finally, in Section \ref{Section: Closed Form Solutions}, we discuss closed form solutions and generating functions for two specific cases of the Schr\"{o}dinger operator.

\section{Preliminaries} \label{Section 2}
\subsection{The Graph Setting}
Let $X$ be a discrete and countable set and let $b: X \times X \to [0,\infty)$ be the \textbf{edge-weight} denoting a symmetric function such that $b(x,x)=0$ for all $x \in X$ and which satisfies the locally summable condition $\sum_{y \in X} b(x,y)  < \infty$ for all $x\in X$. Next we introduce the \textbf{measure} as a function $m: X \to (0, \infty)$ which extends to all subsets of $X$ by additivity. We refer to the triple $(X,b,m)$ as a \textbf{weighted graph}. 
\begin{definition}
    [Birth-Death Chains] A weighted graph is called a \textbf{birth-death chain} if $X = \mathbb{N}_0=\{0,1,2,...\}$ and $b(x,y) >0$ if and only if $|x-y| =1$.
\end{definition} 
Functions on graphs have a natural Hilbert space structure as follows: First, we denote the function space \[ C(X) = \{ f \text{ }| \text{ }  f : X \to \mathbb{R} \}. \] Let $C_c(X)$ denote the space of functions with finite support over $X$, that is, \begin{equation*}
    C_c(X) = \{ f \in C(X) \text{ }| \text{ } f(x) \neq 0 \text{  for finitely many } x \in X \}. 
\end{equation*} 
The space of square summable functions
\begin{equation*}
   \ell^2(X,m) = \{ f \in C(X) \text{ } | \text{ } \sum_{x \in X} f(x)^2m(x) < \infty \}
\end{equation*} 
forms a Hilbert space under the inner product given by $\braket{f,g} = \sum_{ x \in X }  f(x) g(x)m(x)$ for $f,g \in \ell^2(X,m).$ In addition, we note that the space $C_c(X)$ is dense in $\ell^2(X,m)$.
\subsection{Essential Self-Adjointness of Operators on Hilbert Spaces  } \label{Subsection: SelfAdjointness Theory}
We now introduce some basic theory of Hilbert spaces. Let $H$ denote a Hilbert space with inner product given by $\braket{\cdot, \cdot}$.  Let $A$ be a linear map from $\text{Dom}(A) \subseteq H$ to $H$. Here, we refer to $\text{Dom}(A)$ as the \textbf{domain} of the operator $A$. Note that $A$ is not necessarily a bounded operator. Next we define the graph of an operator as follows. 
\begin{definition}
    [Graph of an Operator] Let $A: \text{Dom}(A) \to H$ be an operator on a Hilbert space $H$. We define its \textbf{graph}, denoted as $\Gamma(A)$, as the set \[ \Gamma(A) = \{ (f,Af) \text{ } | \text{ } f \in \text{Dom}(A) \} \subseteq H \times H. \]
\end{definition}
 We say $A$ is a \textbf{closed operator} if $\Gamma(A)$ is closed.  For two operators $A$ and $B$ we say that $B$ is an \textbf{extension} of $A$, or $B \supseteq A$, if $\text{Dom}(A) \subseteq \text{Dom} (B)$ and $B|_{\text{Dom}(A)} = A$. In particular, this implies $\Gamma(A) \subseteq \Gamma(B).$

\begin{definition}
    [Closure of an Operator] Let $A: \text{Dom}(A) \to H$ be an operator on a Hilbert space $H$. We denote $\overline{A}$ as the \textbf{closure of $A$} which is defined as the smallest \textbf{closed} extension of $A$. 
\end{definition}
Consider the case where $A$ has a dense domain $\text{Dom}(A)$. Then we define the \textbf{domain of the adjoint}, or $\text{Dom}(A^*)$, as the set \[ \text{Dom}(A^*) = \{ f \in H \text{ }  | \text{ } h \mapsto \braket{f,Ah} \text{ is a continuous linear functional over Dom}(A)  \}. \] 
By the Riesz Representation Theorem, we can redefine as $\text{Dom}(A^*)$ as 
the set of $f \in H$ for which there exists an $g \in H$ such that \[ \braket{f,Ah} = \braket{g,h} \qquad \text{for all } h \in \text{Dom} (A).\]
We can then define $A^*$, the \textbf{adjoint} of $A$, as the mapping $\text{Dom}(A^*) \to H$ such that $A^* f = g$. In particular, we have \[ \braket{f,Ah} = \braket{A^*f,h} \quad \text{for all } h \in \text{Dom}(A),\text{ } f \in \text{Dom}(A^*). \] 

Next we define symmetric operators. 
\begin{definition}
    [Symmetric and Self-Adjoint Operators] Let $A: \text{Dom}(A) \to H$ be an operator on a Hilbert space $H$ with dense domain. We say $A$ is \textbf{symmetric} if $A^*$ is an extension of $A$. Equivalently, $A$ is symmetric if and only if \[ \braket{Af,g} = \braket{f,Ag} \qquad \text{ for all } f,g \in \text{Dom}(A). \] Furthermore, an operator $A$ is \textbf{self-adjoint} if $A=A^*.$
\end{definition}
Finally, we can define essential self-adjointness as the following: 
\begin{definition}
    [Essential Self-Adjointness]  Let $A: \text{Dom}(A) \to H$ be a symmetric operator on a Hilbert space $H$. We say that $A$ is \textbf{essentially self-adjoint} if $\overline{A}$ is self-adjoint, or equivalently, if $A$ has a unique self-adjoint extension.
\end{definition}
In this note, we will be especially concerned with \textbf{semi-bounded operators}, that is, operators $A$ for which there exists $K \in \mathbb{R}$ such that \[ \braket{Af,f} \geq K \braket{f,f} \qquad \text{for all } f \in \text{Dom}(A).\] We also say that $A$ is \textbf{strictly positive} if the constant $K$ is strictly positive. To prove essential self-adjointness of semi-bounded operators, we will make use of the following theorem from \cite{Reed_Simon_1975}: 
\begin{theorem} \label{Simon Reed}
  If $A$ is a strictly positive symmetric operator, then $A$ is essentially self-adjoint if and only if $\text{Ker}(A^*) = 0$. 
\end{theorem}

\subsection{The Laplacian on Graphs}  \label{Section: Laplacian ESA}
To discuss essential self-adjointness of Schr\"{o}dinger operators on graphs, it is necessary first to introduce the formal Laplacian. In the graph setting, let $\mathcal{F}$ be the set of functions defined by: \begin{equation*}
    \mathcal{F} = \{ f \in C(X) \text{ } | \sum_{y \in X} b(x,y) |f(y)| < \infty \text{ for every } x \in X\}.
\end{equation*}  
Note that for birth-death chains we have $C(X) = \mathcal{F}.$

\begin{definition}
    [Formal Laplacian on Graphs] Let $(X,b,m)$ be a weighted graph. The \textbf{formal Laplacian} is mapping $\Delta: \mathcal{F} \to C(X)$ defined by \begin{equation*}
        \Delta f(x) = \dfrac{1}{m(x)} \sum_{ y \in X} b(x,y) (f(x) - f(y)).
    \end{equation*}
\end{definition}
In the case of birth-death chains, the Laplacian yields a simpler form:
\begin{equation*}
        \Delta f(x) = \dfrac{1}{m(x)} \left( b(x,x+1) \left( f(x) - f(x+1) \right) + b(x,x-1) \left(f(x) - f(x-1) \right) \right).
    \end{equation*} 
   On birth-death chains, when restricted to functions of finite support, the Laplacian maps to square summable functions, i.e.,  $\Delta (C_c(\mathbb{N}_0))  \subseteq \ell^2(\mathbb{N}_0,m)$. Thus, the restriction of the Laplacian on finitely supported functions, $L_c = \Delta|_{C_c(\mathbb{N}_0)}$, is an operator on $\ell^2(\mathbb{N}_0,m)$. A direct calculation then gives the Green's formula: \begin{equation*}
        Q(\varphi, \phi) = \braket{L_c \varphi, \phi} = \braket{\varphi,L_c \phi}  \label{Green's Formula}
    \end{equation*}
    where $\braket{\cdot, \cdot}$ denotes the inner product on $\ell^2(\mathbb{N}_0,m)$ and $Q$ is the energy form given by: \begin{align*}
        Q(\varphi, \phi) &= \dfrac{1}{2} \sum_{x, y \in \mathbb{N}_0} b(x,y) (\varphi(x) - \varphi(y)) (\phi(x)- \phi(y)) \\ 
        &= \sum_{r=0}^{\infty} b(r,r+1) ( \varphi(r) - \varphi(r+1)) (\phi(r) - \phi(r+1)
    \end{align*}
    for $\varphi,\phi \in C_c(\mathbb{N}_0)$. 
    Thus, $L_c$ is a symmetric operator over $C_c(\mathbb{N}_0)$ and, by definition, $L_c$ is essentially self-adjoint if $L_c$ has a unique self-adjoint extension, which is to say the closure $\overline{L_c}$ is self-adjoint. Furthermore, we have that $L_{c}^*$ has domain: \[ \text{Dom} (L_{c}^*) = \{ f \in \ell^2(\mathbb{N}_0,m) \text{ } | \text{ } \Delta f \in \ell^2(\mathbb{N}_0,m) \}. \]
    
    By general theory \cite{Reed_Simon_1975}, the essential self-adjointness of $L_c$ is equivalent to the triviality of $\lambda$-harmonic functions, that is to say functions $f \in \mathcal{F}$ such that \begin{equation*}
        \Delta f = \lambda f \qquad \text{for } \lambda < 0. 
    \end{equation*}
    For $\lambda=0,$ we call such functions \textbf{harmonic.} Thus, $L_c$ is essentially self-adjoint if and only if  \begin{equation*}
        \text{Ker} (L_{c}^* - \lambda) = \{0\}. 
    \end{equation*}
    for one (or all) $\lambda < 0$. The essential self-adjointness of the Laplacian on birth-death chains is a classically known result by Hamburger \cite{Ham20}, with a novel proof available in \cite{IKMW}:  
    \begin{theorem}\label{Hamburger's Criterion}
        Let $(\mathbb{N}_0,b,m)$ be a birth death chain. Then the following statements are equivalent: 
        \begin{enumerate}
            \item $L_c$ is essentially self-adjoint. 
            \item $\displaystyle \sum_{k=0}^{\infty} \left( \sum_{r=0}^k \dfrac{1}{b(r,r+1)} \right)^2 m(k+1) = \infty$.
        \end{enumerate}
    \end{theorem} 
We note that essential self-adjointness can also be characterized by a notion of capacity $\text{Cap}_{2,2}(\infty)$ which requires the construction of second order Sobolev spaces on graphs (see \cite{IKMW} for details).
    \subsection{Schr\"{o}dinger Operators on Birth-Death Chains}  \label{Section: Schrodinger Theory}
   
    \begin{definition}
        [Formal Schr\"{o}dinger Operators on Birth-Death Chains] Let $(X,b,m)$ be a weighted-graph. We denote $\Delta + W$ as the \textbf{formal Schr\"{o}dinger operator} where $\Delta$ is the formal Laplacian as defined in Definition 2.7 and $W \in C(X)$ is the\textbf{ potential function} acting on functions via $Wv(x) = W(x) v(x)$.
    \end{definition}

 As in the case of the Laplacian on birth-death chains, we define the \textbf{Schr\"{o}dinger operator} $L_c+ W$ as the formal Schr\"{o}dinger operator acting on $C_c(\mathbb{N}_0)$: \[ L_c +W : C_c(\mathbb{N}_0) \to \ell^2(\mathbb{N}_0,m). \] First we discuss known results for essential self-adjointness of the Schr\"{o}dinger operators and its connection to the Laplacian \cite{IKMW}. The first result states that adding a lower bounded potential does not impact essential self-adjointness of the Laplacian: 
    \begin{proposition}
       Let $(\mathbb{N}_0,b,m)$ be a birth-death chain. Let $W \in C(\mathbb{N}_0)$ with $W(x) \geq K$ for all $x \in \mathbb{N}_0$ for some $K \in \mathbb{R}$. If $L_c$ is essentially self-adjoint, then $L_c+ W$ is essentially self-adjoint.
    \end{proposition}
    Thus, for lower bounded potentials, Hamburger's Criterion (Theorem \ref{Hamburger's Criterion}) can be utilized to give a criterion for essential self-adjointness. We note that the above proposition is not an equivalence statement, the next proposition from \cite{IKMW} gives an equivalence statement showing that essential self-adjointness of Schr\"{o}dinger operators can be evaluated with a positive solution: 
     \begin{proposition}
          Let $(\mathbb{N}_0,b,m)$ be a birth-death chain and $W \in C(\mathbb{N}_0)$ such that $L_c+ W$ is semi-bounded operator with lower bound $K \in \mathbb{R}$. Let $v \in C(\mathbb{N}_0)$ with $v>0$ satisfy $(\Delta + W) v  = \lambda v$ for $\lambda < K$. Then $L_c + W$ is essentially self-adjoint if and only if \begin{equation*}
            \sum_{r=0}^{\infty} \left( \sum_{k=0}^r \dfrac{1}{b(k,k+1) v(k)v(k+1)} \right)^2 v^2(r+1) m(r+1) = \infty.
        \end{equation*}
     \end{proposition}
    Note that while the above proposition gives an equivalence, it does not address what a function $v > 0$ satisfying $(\Delta + W) v = \lambda v$ looks like. We address this question in the next section. \smallskip
    
     Motivated by the discussion in Section \ref{Subsection: SelfAdjointness Theory} and Theorem \ref{Simon Reed}, throughout the note we will assume that $\Delta + W$ yields a semi-bounded operator with lower bound $K \in \mathbb{R}$. Explicitly this means that there exists a $K \in \mathbb{R}$ such that  \begin{equation*}\label{Strict Positive Condition}
 \braket{ (\Delta + W ) \varphi  , \varphi }   \geq K \norm{\varphi}^2
\end{equation*}
for all $\varphi \in C_c(X)$. From our discussion in Section \ref{Section: Laplacian ESA} and Theorem \ref{Simon Reed}, we have the following equivalent definition for essential self-adjointness of semi-bounded Schr\"{o}dinger operators on birth-death chains: \smallskip 

     If $L_c + W$ is a semi-bounded Schr\"{o}dinger operator with lower bound $K \in \mathbb{R}$ on a birth-death chain, then $L_c + W$ is essentially self-adjoint if and only if the only functions $ v \in \ell^2(\mathbb{N}_0,m)$ satisfying $(\Delta + W - \lambda) v =0$ is the trivial function $v \equiv 0$ for some $\lambda < K$.

\section{Characterization of Essential Self-Adjointness from Solving Linear Recurrence}  \label{Section: Characterization of ESA from Solving Linear Recurrence}
As Section \ref{Section: Schrodinger Theory} suggests, the key to characterizing essential self-adjointness of Schr\"{o}dinger operators lies in deriving explicit formulas for $v \in C(\mathbb{N}_0)$ satisfying $(\Delta + W - \lambda) v(r) =0$ and studying their norm in $\ell^2(\mathbb{N}_0,m)$. While closed form solutions do exist,  they yield convoluted formulas. Nonetheless, these formulas allow for a full characterization of essential self-adjointness for semi-bounded Schr\"{o}dinger operators which we offer in this section. \smallskip 

In \cite{Mallik_1998}, a complete solution is derived in terms of the variable coefficients, which we now present:
\begin{theorem}  \label{Alt Long-History} Let $(\mathbb{N}_0,b,m)$ be a birth-death chain and $W \in C(\mathbb{N}_0)$ be a potential function such that $L_c + W$ is semi-bounded with lower bound $K$ and let $\lambda <K$. If $v: \mathbb{N}_0 \to \mathbb{R}$ such that $(\Delta  + W - \lambda) v(k) = 0$ for all $k \in \mathbb{N}_0$, then  \[ v(k+1) = (\alpha_k + \beta_k ) v(0) \] for all $k \geq 0$ where $\alpha_{k}$ and $\beta_k$ are given by
     \begin{align*}
        \alpha_{k} &=  \Theta \sum_{r=1}^k \sum_{\substack{ (\ell_1,\ell_2,..., \ell_r) \\ \ell_j \in \{1,2\} \\ \ell_1 + ... + \ell_r = k}} \left( \prod_{ \{ \ell_i |\ell_i =1\}} \left( 1 + \dfrac{b( c_{k,i}, c_{k,i} + 1 ) +   ( W(c_{k,i} + 1) - \lambda) m( c_{k,i} + 1 )}{b(c_{k,i} + 1, c_{k,i} + 2 )}\right)\right. \\ & \left.  \prod_{ \{ \ell_i |\ell_i =2\}} \left( \dfrac{-b( c_{k,i} + 1 ,  c_{k,i} +  2) }{b(c_{k,i} +  2,  c_{k,i} +  3)} \right)\right)   \\
         \beta_{k} &= \sum_{r=1}^{k+1} \sum_{\substack{ (\ell_1,\ell_2,..., \ell_r) \\ \ell_j \in \{1,2\} \\ \ell_1 + ... + \ell_r = k+1 \\ \ell_r =2}} \left( \prod_{ \{ \ell_i |\ell_i =1\}} \left( 1 + \dfrac{b( c_{k,i}, c_{k,i} + 1 ) +   ( W(c_{k,i} + 1) - \lambda) m( c_{k,i} + 1 )}{b(c_{k,i} + 1, c_{k,i} + 2 )}\right)\right. \\ & \left.  \prod_{ \{ \ell_i |\ell_i =2\}} \left( \dfrac{-b( c_{k,i} + 1 ,  c_{k,i} +  2) }{b(c_{k,i} +  2,  c_{k,i} +  3)} \right)\right) 
    \end{align*} 
   for $k \geq 1$ with $\Theta =  1+ \dfrac{(W(0) - \lambda)}{b(0,1)}m(0) $ and $c_{k,i} ( ( \ell_n ) )= c_{k,i} = k - \sum_{n=1}^i \ell_n $. For $k=0,$ we let $ \alpha_0 = \Theta $ and $ \beta_0 =0$.
    \end{theorem} 
    \begin{remark}
    Besides the difference of summing to $k$ and $k+1$ respectively for the coefficients in $\alpha_k$ and $\beta_k$, $\beta_k$ includes the additional restriction in its second sum where we require $\ell_r = 2.$
    \end{remark}
    With this formula, we can classify the essential self-adjointness of semi-bounded Schr\"{o}dinger operators by studying the norm of solutions in $\ell^2(\mathbb{N}_0,m)$.

      \begin{theorem}[Characterization of Essential Self-Adjointness of Semi-bounded Schr\"{o}dinger Operators on Birth-death Chains] \label{Theorem: Characterization of ESA }  Let $(\mathbb{N}_0,b,m)$ be a birth-death chain and $W \in C(\mathbb{N}_0)$ be a potential function such that $L_c + W$ is semi-bounded with lower bound $K$. Let $\alpha_{k}, \beta_{k}$ be defined as in Theorem \ref{Alt Long-History}. Then, the following statements are equivalent: 
    \begin{enumerate}
        \item $\displaystyle \sum_{k=0}^{\infty}   \left( \alpha_k + \beta_k \right)^2m(k+1) = \infty $ for some $\lambda < K$. 
        \item  $L_c + W$ is essentially self-adjoint.
    \end{enumerate} 
        \end{theorem} 
    \begin{proof} \smallskip
      
          \qquad $(1) \implies (2)$: Let $v \in \ell^2(\mathbb{N}_0,m)$ be a solution to $(\Delta + W-\lambda)v(r) =0$. We will show that $v \equiv 0$. First we use the expression from Theorem \ref{Alt Long-History} to calculate its $\ell^2$ norm:  \begin{align*}
                \sum_{k=0}^{\infty} v(k+1)^2m(k+1) 
                &= v^2(0) \sum_{k=0}^{\infty} \left( \alpha_k + \beta_{k} \right)^2m(k+1) < \infty
            \end{align*} 
            However, by assumption \[ \sum_{k=0}^{\infty} \left( \alpha_k + \beta_{k} \right)^2m(k+1) = \infty. \] Hence we must have $v(0)=0$, and thus 
           Theorem \ref{Alt Long-History} gives $v \equiv 0$. 
             \smallskip 
            
            \qquad $(2) \implies (1)$: We proceed by contraposition, let \[\sum_{k=0}^{\infty} \left( \alpha_k + \beta_{k} \right)^2m(k+1) < \infty.\] Then by the argument above, any $v(0)\neq0$ induces a non-trivial solution $v \in \ell^2(\mathbb{N}_0,m)$ to $(\Delta + W-\lambda)v =0$. Hence, $L_c +W$ is not essentially self-adjoint.
    \end{proof}
    \begin{remark}
        In view of Hamburger's Theorem \ref{Hamburger's Criterion}, by setting $W \equiv 0$, we obtain the nontrivial fact that \[ \sum_{k=0}^{\infty} \left( \sum_{r=0}^k \dfrac{1}{b(r,r+1)} \right)^2 m(k+1) = \infty\] if and only if \[\sum_{k=0}^{\infty} \left(  \alpha_k + \beta_k \right)^2m(k+1) = \infty  .\]
    \end{remark}
    As we see from the above representations of Schr\"{o}dinger operator solutions are algebraicaly complex and impractical in evaluating the essentially self-adjointness. Thus, we are motivated to explore alternate statements to study this problem. In the following section, we derive a condition for the failure of essential self-adjointness that is much easier to check.
     \section{Failure of Essential Self-Adjointness of Schr\"{o}dinger Operators  } \label{Section: Additional Results}
We present an additional result that inspects the failure of essential self-adjointness of Schr\"{o}dinger operators that has a simpler statement than Theorem \ref{Theorem: Characterization of ESA }. 
\begin{theorem}
  Let $(\mathbb{N}_0,b,m)$ be a birth-death chain and $W \in C(\mathbb{N}_0)$ be a potential function such that $L_c + W$ is semi-bounded with lower bound $K.$ If 
     \begin{align*}
          \sum_{k=0}^{\infty} \left( \prod_{i=0}^{k} \left[  \left( 1+ \dfrac{b(i,i+1) +(W(i+1)-\lambda)m(i+1)}{b(i+1,i+2)}  \right)^2  \right. \right. \\ + \left. \left. \left( \dfrac{b(i,i+1)}{b(i+1,i+2)}  \right)^2  +1 \right]\right)  m(k+2) < \infty  \numberthis  \label{Alt ESA}
     \end{align*}
     for some (all) $\lambda <K$, then $L_c + W$ is not essentially self-adjoint. 
\end{theorem}
\begin{proof} Let (\ref{Alt ESA}) hold true, we will show that any $v(r)$ such that $v(0)\neq 0$ and $(\Delta + W - \lambda) v =0$ induces a non-trivial function in $\ell^2(\mathbb{N}_0,m)$. \smallskip 

   Denoting $\alpha_i = \dfrac{(W(i+1)-\lambda)m(i+1)}{b(i+1,i+2)},  \text{ } \beta_i = \dfrac{b(i,i+1)}{b(i+1,i+2)}$. We define a two by two matrix $A_i$ and vector $\Vec{y_i}$ by
    \begin{align*}
        A_i &\coloneqq  \begin{bmatrix}
            1+ \alpha_i + \beta_i &  -\beta_i \\ 1 &0
        \end{bmatrix} \\ 
        \Vec{y_i} & \coloneqq \begin{bmatrix}
            v(i) \\ v(i-1)
        \end{bmatrix},
    \end{align*}
We find that the recurrence arising from evaluating the formula $(\Delta + W - \lambda) v(r) = 0 $ at each point is equivalent to the matrix recurrence \begin{align*}
        \Vec{y_2} &= A_0 \Vec{y_1} \\ 
        \Vec{y_3} &= A_1 \Vec{y_2} = A_1 A_0 \Vec{y_1} \\ 
        & \vdots \\ \Vec{y_k} &= \prod_{i=0}^{k-2} A_i\Vec{y_1}. \numberthis \label{2 norm equation}
    \end{align*}
    Here the product operator $\prod$ denotes left multiplication of matrices. In the following calculation we denote by $\norm{\cdot}_2$ the 2-norm on $\mathbb{R}^2$, i.e., \[ \norm{x}_2 = \sqrt{x_1^2 + x_2^2}, \] and  $\norm{\cdot}_{\RN{2}}$ denotes the operator norm on the space of matrices, that is, \[ \norm{A}_{\RN{2}} = \sup_{ \norm{x}_2 \neq 0} \dfrac{\norm{Ax}_2}{\norm{x}_2}. \] 
    By evaluating the 2-norm on both sides of (\ref{2 norm equation}) we obtain
    \begin{align*}
        \norm{\vec{y}_k}_2 &=  \norm{ \prod_{i=0}^{k-2} A_i \Vec{y_i}}_2 \\
        &\leq \norm{\prod_{i=0}^{k-2} A_i}_{\RN{2}} \sqrt{v(1)^2+v(0)^2} \\ & \leq \prod_{i=0}^{k-2} \norm{A_i}_{\RN{2}} \sqrt{v(1)^2 + v(0)^2}.  \numberthis \label{Sing Vales}
    \end{align*} By general theory in linear algebra, the operator norm of any matrix $B$ is equal to its largest singular value, that is, \[ \norm{B}_{\RN{2}} = \sigma_{\text{max}}(B)  = \sqrt{\lambda_{\text{max}} (B^* B)},\] where $\lambda_{\text{max}} (B^{*} B)$ is the largest eigenvalue of $B^{*}B$ (see \cite{YangLinAlg} for details). Thus we apply this estimate to (\ref{Sing Vales}) to get 
    \begin{align*}
        \norm{\vec{y}_k}_2  
        &\leq \prod_{i=0}^{k-2} \sqrt{\lambda_{\text{max}}(A_i^{*} A_i) }\sqrt{v(1)^2+v(0)^2}.
    \end{align*}
 We calculate that \begin{align*}
        A_i^{*} A_i &= \begin{bmatrix}
            (1+ \beta_i + \alpha_i)^2+1 & - (\beta_i+\beta_i^2 + \beta_i\alpha_i) \\- (\beta_i+\beta_i^2 + \beta_i\alpha_i) & \beta_i^2
        \end{bmatrix}. 
    \end{align*}We note that $A_i^{*} A_i$ is symmetric, hence, it only contains real eigenvalues. Since the trace is the sum of eigenvalues, we have that \[ \lambda_{\text{max}} (A_i^{*} A_i) \leq \text{Trace}(A_i^{*} A_i) = (1+\beta_i + \alpha_i)^2 + \beta_i^2 + 1. \]  Hence 
    \begin{align*}
        |v(k)| \leq  \sqrt{v(k)^2 + v(k-1)^2} \leq \prod_{i=0}^{k-2} \sqrt{ (1+ \beta_i + \alpha_i)^2+ \beta_i^2 + 1} \sqrt{v(1)^2 + v(0)^2}.
    \end{align*}
    Taking the square, multiplying by the measure and taking the sum to get 
    \begin{align*}
        \sum_{k=2}^{\infty} \left( \prod_{i=0}^{k-2} \left[  \left( 1+ \dfrac{b(i,i+1) + (W(i+1)-\lambda)m(i+1)}{b(i+1,i+2)} \right)^2 \right. \right.  +  \\ 
     \left. \left. \left( \dfrac{b(i,i+1)}{b(i+1,i+2)} \right)^2  + 1\right]\right)  m(k) < \infty.
    \end{align*}  
    By a shift of index we have proven the theorem.
\end{proof}

    \section{Closed Form Representations of Solutions for the Schr\"{o}dinger Operator} \label{Section: Closed Form Solutions}
   In this last section, we explore other methods to study closed form solutions for two specific cases of the potential functions. We will show that the second order recurrence that arises from studying the Schr\"{o}dinger operator problem is equivalent to a full-history recurrence, i.e., a recurrence relation that takes the form $x_k = F(x_1,...,x_{k-1})$ for all $k \in \mathbb{N}$. It can be shown that, for special cases, this full-history recurrence representation of solution yields closed form generating functions and series representations (for more on generating functions, see \cite{Wilf}). Note that, as the contents of this section are independent of other sections of this note, we will omit the $W(r) -\lambda$ term from the previous section and replace it with $W(r)$ for the sake of simplicity.

Full-history recurrences have been extensively studied. We especially note that in \cite{Mallik_1998}, there exists a closed form representation for inhomogeneous full-history recurrences in terms of the driving force terms, however, these solutions evaluate to the trivial sequence when specialized to the homogeneous case. In the current literature, there are no expressions in terms of homogeneous full-history recurrences, however, we will show that by considering special cases of Schr\"{o}dinger operators, more digestible closed form solutions can be derived. 

          \begin{lemma} \label{lemma: full history recurrence}  Let $(\mathbb{N}_0,b,m)$ be a birth-death chain. Then $v: \mathbb{N}_0 \to \mathbb{R}$ such that $(\Delta  + W) v(r) = 0$ is a solution to the linear full-history recurrence  \begin{align}
     v(r+1) = v(1) + C \sum_{k=1}^r \dfrac{1}{b(k,k+1)} +   \sum_{k=1}^r \left( \dfrac{1}{b(k,k+1)} \left( \sum_{n=1}^k v(n) W(n)m(n) \right)\right) \label{1.5}
\end{align} for all $r \geq 0$, where $C = b(1,0)(v(1) -v(0))$. \end{lemma}  
\begin{remark}
    We note that in the case $W \equiv 0$, the recurrence relation takes the form of closed form solutions for functions which are harmonic except at $r=0$. In \cite{IKMW}, this fact is used in proving Hamburger's Criterion (Theorem \ref{Hamburger's Criterion}).
\end{remark}  
\begin{proof}
    The proof follows from a strong induction argument \end{proof} 
    Next, we show the full history recurrences in Lemma \ref{lemma: full history recurrence} can be expressed as another full history recurrence in terms of the generating function $A(z) = \sum_{r=1}^{\infty} v_r z^r$ where the $r$ index for $v$ is now placed in the lower index to emphasize the treatment of $v$ as a sequence over $\mathbb{N}_0$. That is, $v_r = v(r)$ with similar expression for the potential function $W$ and measure $m$.
    
\begin{theorem} \label{Theorem generating function}
    Let $A(z)$ denote the generating function $A(z) = \sum_{r=1}^{\infty} v_r z^r$ where $v_r$ satisfies the recurrence relation (\ref{1.5}). If $v_r \neq 0$ for all $r \geq 0$, then $A(z)$ is given by
      \[ A(z) \left( 1- \sum_{n=1}^{\infty} z^n  \left( \sum_{r=1}^{\infty} \beta_{r,n} z^r \right) \right) = C \sum_{r=1}^{\infty} \left(  \sum_{k=1}^{r-1} \dfrac{1}{b(k,k+1)} \right) z^r +   \dfrac{v_1z}{1-z}. \] 
    where $\beta_{r,n}$ is defined by the relation\begin{equation}
        W_rm_rv_r \sum_{\ell =r}^{r+n-1} \dfrac{1}{b(\ell, \ell +1)}   = \sum_{\gamma=1}^r \beta_{r-\gamma, n} v_{\gamma}.  \label{Relation}
    \end{equation}
\end{theorem}  
\begin{remark}
    The assumption that $v_r \neq 0$ for all $r \geq 0$ is required so that the relation (\ref{Relation}) is well defined. As the reader can verify later, this assumption is not required in proving Corollaries \ref{Corollary 1} and \ref{Corollary 2}.
\end{remark}
\begin{proof}
    We first commute the summation indices in the recurrence (\ref{1.5}) to obtain
\begin{align*}
    v_r &= v_1 + C \sum_{k=1}^{r-1} \dfrac{1}{b(k,k+1)} + \sum_{k=1}^{r-1} v_kW_km_k \left( \sum_{n=k}^{r-1} \dfrac{1}{b(n,n+1)} \right) 
\end{align*}
for all $r \geq 1$. Multiply both sides by $z^r$ and summing to $r=\infty$ to get
\begin{align*}
    \sum_{r=1}^{\infty} v_rz^r &= \sum_{r=1}^{\infty} v_1 z^r + \sum_{r=1}^{\infty}\left( C \sum_{k=1}^{r-1} \dfrac{1}{b(k,k+1)} + \sum_{k=1}^{r-1} \left( \sum_{n=k}^{r-1} \dfrac{1}{b(n,n+1)} \right) v_k W_km_k \right) z^r \\ 
     &=\dfrac{v_1z}{1-z} + C \sum_{r=1}^{\infty}\left( \sum_{k=1}^{r-1} \dfrac{1}{b(k,k+1)} \right) z^r  \\ &+ \sum_{n=1}^{\infty} z^n \left( \sum_{r=1}^{\infty} W_rm_r \left( \sum_{\ell = r}^{r+n-1} \dfrac{1}{b(\ell, \ell + 1)} \right) v_r z^r \right) .
\end{align*}
By the Cauchy product formula, we decompose the third term on the right hand side to get
\begin{align*}
    \left( \sum_{r=1}^{\infty} W_rm_r \left( \sum_{\ell = r}^{r+n-1} \dfrac{1}{b(\ell, \ell + 1)} \right) v_r z^r \right)  = \left( \sum_{r=1}^{\infty} v_r z^r \right) \left(  \sum_{r=1}^{\infty} \beta_{r,n}z^r\right) = A(z) \cdot \sum_{r=1}^{\infty} \beta_{r,n}z^r,
\end{align*}
where $\beta_{r,n}$ is defined by \[ W_rm_r v_r \sum_{\ell =r}^{r+n-1} \dfrac{1}{b(\ell, \ell +1)}  = \sum_{\gamma=1}^r \beta_{r-\gamma, n} v_{\gamma}. \] Putting everything together,
\begin{align*}
    A(z) =  \dfrac{v_1z}{1-z} +  C \sum_{r=1}^{\infty} \left( \sum_{k=1}^{r-1} \dfrac{1}{b(k,k+1)} \right)  z^r  + \sum_{n=1}^{\infty} z^n A(z) \left( \sum_{r=1}^{\infty} \beta_{r,n} z^r \right)
\end{align*} 
which proves the theorem.
\end{proof}
With these tools, we present two special cases of the edge-weight and the potential function such that a closed form solution for the generating function can be derived. In the first case we consider a constant edge-weight $b(r,r+1) = b$ and product between potential and measure $W(r)m(r) = c$, and derive its generating function and a surprisingly complex power series representation. 
\begin{corollary}  \label{Corollary 1} \label{Theorem: Wolfram Alpha}Let $(\mathbb{N}_0,b,m)$ be a birth-death chain with $b(r,r+1) = b$ and $ W(r)m(r) = c$ for all $r \geq 0$. If $v: \mathbb{N}_0 \to \mathbb{R}$ is a function such that $(\Delta  + W) v = 0$, then the generating function $A(z) =  \sum_{r=1}^{\infty} v_rz^r$ has an explicit form \begin{align}
         A(z) &=  \dfrac{(1-z)^2}{(1-z)^2- \alpha z} \left( \dfrac{(v_1-v_0)z}{(1-z)^2}  + \dfrac{v_0z}{1-z} \right),
         \end{align}
         where $\alpha = c/b$.
\end{corollary} 
\begin{remark}
    We note that, by substituting the generating function in Corollary \ref{Theorem: Wolfram Alpha} into Wolfram Alpha, we find that it gives a power series representation \[ v_r = 
    \dfrac{2^{-r}}{\sqrt{  \beta_{\alpha} }} \left( 2 \left(  \left( \xi_{\alpha} -  \beta_{\alpha} \right)^{-1+r} -  \left( \xi_{\alpha} + \beta_{\alpha} \right)^{-1 + r} \right) v_0 +  \left( - \left(\xi_{\alpha} - \beta_{\alpha} \right)^r + \left( \xi_{\alpha} + \beta_{\alpha} \right)^r  \right)v_1  \right)
\]
for all $r \geq 3$, where $\beta_{\alpha} =  \sqrt{ 4  \alpha + \alpha^2}$ and $\xi_{\alpha} = 2 + \alpha$. 
\end{remark}
\begin{proof}
    Following the proof of Theorem \ref{Theorem generating function} and recalling that $C = b(v_1-v_0)$, we arrive at the relation \begin{align*}
    A(z) &= \sum_{r=1}^{\infty} v_rz^r \\ 
    &= \sum_{r=1}^{\infty} v_1 z^r + \sum_{r=1}^{\infty}\left( C  \left( \sum_{k=1}^{r-1} \dfrac{1}{b(k,k+1)} \right)  + \sum_{k=1}^{r-1} \left( \sum_{n=k}^{r-1} \dfrac{1}{b(n,n+1)} \right) v_k W_km_k \right) z^r 
    \\ &= 
    \sum_{r=1}^{\infty} v_1 z^r + \sum_{r=1}^{\infty} b(v_1-v_0) \dfrac{(r-1)}{b} z^r 
    + \sum_{n=1}^{\infty} z^n\left( \sum_{r=1}^{\infty} c \left( \sum_{\ell =r}^{r+n-1} \dfrac{1}{b} \right) v_r z^r \right) \\ &=   \sum_{r=1}^{\infty} v_1 z^r + \sum_{r=1}^{\infty} \left( (v_1-v_0)r - (v_1-v_0)\right) z^r +\sum_{n=1}^{\infty} z^n \left( \sum_{r=1}^{\infty}\dfrac{cn}{b}v_r z^r \right) \\ &=  (v_1-v_0)\sum_{r=1}^{\infty} r z^r + v_0\sum_{r=1}^{\infty}z^r + \left( \sum_{n=1}^{\infty}\alpha n z^n \right)  \left( \sum_{r=1}^{\infty} v_r z^r \right) \\ &= \dfrac{(v_1-v_0)z}{(1-z)^2} + \dfrac{v_0z}{1-z} + \dfrac{\alpha z}{(1-z)^2} A(z).
     \end{align*}
     We obtain the generating function by solving for $A(z)$. 
\end{proof} 
Next we consider the case of constant edge-weight but the product of Schr\"{o}dinger operator and measure satisfies a power law. Unlike the previous case, there is not a closed form solution for the generating function in general, however, we will show it satisfies a ordinary differential equation.
\begin{corollary}  \label{Corollary 2}
Let $(\mathbb{N}_0,b,m)$ be a birth-death chain with $b(r+1,r) = b$ for all $r \geq 0$. Assume that $W_rm_r = r^{\gamma}$ for some $\gamma \in \mathbb{N}$. If $v: \mathbb{N}_0 \to \mathbb{R}$ is a function such that $(\Delta  + W) v = 0$, then the generating function $A(z) = \sum_{r=1}^{\infty} v_r z^r$ is a solution to the $\gamma$th order differential equation  \[ A(z) = \dfrac{v_0z}{1-z} +  \dfrac{(v_1 - v_0)z}{(1-z)^2}+ \left( \dfrac{z/b}{(z-1)^2} \right) \sum_{k=0}^{\gamma} { \gamma \brace k} \left( \left( \dfrac{d}{dz} \right)^k A(z) \right) z^k, \]  
    where $ \displaystyle { \gamma \brace k} = \sum_{i=0}^k \dfrac{(-1)^{k-i} i^{\gamma} }{(k-i)!i!}$ are the Stirling numbers of second kind.
   
\end{corollary}
\begin{remark}
    In particular, if $\gamma=1$, the solution is given by \[ A(z) = \exp{\left( \frac{z^2 - 1}{bz} \right)}  \dfrac{z^{- 2/b} }{b}\int^z \exp{ \left( \dfrac{1- \xi^2}{b \xi} \right) } \xi^{\frac{2}{b}-1} (\xi v_0 - v_1) d \xi .\]  
\end{remark} 
\begin{proof}
  We once again the follow the proof of Theorem \ref{Theorem generating function} to obtain the relation   
\begin{align*}
    A(z) &= \sum_{r=1}^{\infty} v_1 z^r + C \sum_{r=1}^{\infty} \left( \sum_{k=1}^{r-1} \dfrac{1}{b} \right) z^r +  \dfrac{1}{b} \cdot \left( \sum_{n=1}^{\infty} n z^n \right)  \left( \sum_{r=1}^{\infty}r^{\gamma} v_r z^r \right) \\ 
    &=  \dfrac{v_1 z}{1-z} + \sum_{r=1}^{\infty}  (b(v_1-v_0) \dfrac{(r-1)}{b}) z^r  +  \dfrac{1}{b} \cdot \left( \sum_{n=1}^{\infty} n z^n \right)  \left( \sum_{r=1}^{\infty}r^{\gamma} v_r z^r \right) \\
    &=  \dfrac{v_0z}{1-z} +  \dfrac{(v_1 - v_0)z}{(1-z)^2}+ \left( \dfrac{z/b}{(z-1)^2} \right) \left( \sum_{r=1}^{\infty}r^{\gamma} v_r z^r \right).
\end{align*} 
Let $(r)_k$ denote the falling factorial \[ (r)_k = \underbrace{r(r-1) \cdots (r-k+1)}_{k\text{ terms}}.\]
We can decompose $r^{\gamma}$ in terms of falling factorials with \[ r^{\gamma} = \sum_{k=0}^{\gamma}  {\gamma \brace k }  (r)_k  \] (see \cite{Wilf} for a detailed discussion on Stirling numbers).  
Substituting back to the equation yields 
\begin{align*}
    A(z) &=\dfrac{v_0z}{1-z} +  \dfrac{(v_1 - v_0)z}{(1-z)^2}+ \left( \dfrac{z/b}{(z-1)^2} \right)   \sum_{k=0}^{\gamma} { \gamma \brace k} \sum_{r=1}^{\infty} (r)_k v_r z^{r-k}z^k \\ 
    &=  \dfrac{v_0z}{1-z} +  \dfrac{(v_1 - v_0)z}{(1-z)^2}+ \left( \dfrac{z/b}{(z-1)^2} \right) \sum_{k=0}^{\gamma}  \sum_{r=1}^{\infty} { \gamma \brace k} (r)_k v_r z^{r-k}z^k \\
    &= \dfrac{v_0z}{1-z} +  \dfrac{(v_1 - v_0)z}{(1-z)^2}+ \left( \dfrac{z/b}{(z-1)^2} \right) \sum_{k=0}^{\gamma} { \gamma \brace k} \left( \left( \dfrac{d}{dz} \right)^k A(z) \right) z^k .
\end{align*}
This completes the proof.
\end{proof}

    \section*{Acknowledgments} 
    This research was conducted as part of the Queens Experience in Discrete Mathematics REU. The author gratefully acknowledges the support of mentor Prof. Rados{\l}aw K. Wojciechowski and program director Prof. Rishi Nath.

\end{document}